\documentclass[12pt]{article}

\usepackage{amstext    }
\usepackage{amsthm    }
\usepackage{a4}
\usepackage[mathscr]{eucal}
\usepackage{mathrsfs}

\usepackage{amsmath}
\usepackage{amssymb}
\usepackage{amscd}
\newtheorem{theorem}{Theorem}[section]
\newtheorem{definition}[theorem]{Definition}
\newtheorem{proposition}[theorem]{Proposition}
\newtheorem{corollary}[theorem]{Corollary}
\newtheorem{lemma}[theorem]{Lemma}
\newtheorem{remark}[theorem]{Remark}

\newtheorem{problem}[theorem]{Problem}

\newcommand{\cali}[1]{\mathscr{#1}}

\newcommand{\U}{{\rm U}}

\newcommand{\ddc}{dd^c}

\newcommand{\prim}{{\rm prim}}
\newcommand{\dbar}{\overline\partial}

\newcommand{\Ec}{\cali{E}}

\newcommand{\C}{\mathbb{C}}

\newcommand{\R}{\mathbb{R}}

\newcommand{\transposee}[1]{{\vphantom{#1}}^{\mathit t}{#1}}

\title{On the Lefschetz and Hodge-Riemann theorems}

\author{Tien-Cuong Dinh and Vi{\^e}t-Anh Nguy{\^e}n}

\begin{document}

\maketitle

\begin{abstract}
We give an abstract version of the hard Lefschetz theorem, 
the Lefschetz decomposition and the Hodge-Riemann theorem 
for compact K{\"a}hler manifolds. 
\end{abstract}

\noindent
{\bf Classification AMS 2010}: 32Q15, 14C30, 58A14.


\section{Introduction} \label{introduction}

Let $X$ be a compact K{\"a}hler manifold of dimension $n$ and let $\omega$ be a
K{\"a}hler form on $X$. Denote by $H^{p,q}(X,\C)$
the Hodge cohomology group of bidegree $(p,q)$ of $X$ with the
convention that $H^{p,q}(X,\C)=0$ unless $0\leq p,q\leq n$. 
When $p,q\geq 0$ and $p+q\leq n$, 
put $\Omega:=\omega^{n-p-q}$ and define a Hermitian 
form $Q$ on $H^{p,q}(X,\C)$ by
$$Q(\{\alpha\},\{\beta\}):= i^{q-p}(-1)^{(p+q)(p+q+1)\over 2}
\int_X\alpha\wedge\overline\beta\wedge\Omega$$
for smooth closed $(p,q)$-forms $\alpha$ and $\beta$. The last integral depends only on the classes $\{\alpha\}, \{\beta\}$ of $\alpha, \beta$ in $H^{p,q}(X,\C)$. 

The classical Hodge-Riemann theorem asserts that $Q$ is positive-definite on the primitive subspace 
$H^{p,q}(X,\C)_\prim$ of $H^{p,q}(X,\C)$ which is given by
$$H^{p,q}(X,\C)_\prim:=\big\{\{\alpha\}\in H^{p,q}(X,\C),
\quad \{\alpha\}\smile \{\Omega\}\smile \{\omega\}=0\big\},$$
where $\smile$ denotes the cup-product on the cohomology 
ring $\oplus H^*(X,\C)$, see e.g. Demailly \cite{bdi}, 
Griffiths-Harris \cite{GriffithsHarris} and Voisin \cite{Voisin}. 

The Hodge-Riemann theorem implies the hard Lefschetz theorem which says that the linear map $\{\alpha\}\mapsto \{\alpha\}\smile \{\Omega\}$ defines an isomorphism between $H^{p,q}(X,\C)$ and $H^{n-q,n-p}(X,\C)$. It also implies the following Lefschetz decomposition
$$H^{p,q}(X,\C)=\{\omega\}\smile H^{p-1,q-1}(X,\C)\oplus H^{p,q}(X,\C)_\prim$$
which is orthogonal with respect to the Hermitian form
$Q$.  
Moreover, we easily obtain from the above theorems the signature of $Q$ in term of the Hodge numbers $h^{p,q}:=\dim H^{p,q}(X,\C)$. For example, when $p=q=1$ the signature of $Q$ is equal to $(h^{1,1}-1,1)$.  

The above three theorems are not true if we replace $\{\Omega\}$ with an arbitrary class in 
$H^{n-p-q,n-p-q}(X,\R)$, even when the class contains a strictly positive form, see e.g. Berndtsson-Sibony \cite[\S 9]{BerndtssonSibony}.
Our aim here is to give sufficient conditions on 
$\{\Omega\}$ for which 
these theorems still hold. We will say that such a class $\{\Omega\}$ satisfies the Hodge-Riemann theorem, the hard Lefschetz theorem and the Lefschetz decomposition theorem for the bidegree $(p,q)$.

If $E$ is a complex vector space of dimension $n$ and $\overline E$
its complex conjugate, we will introduce in the next section the
notion of  Hodge-Riemann cone in the exterior product
$\bigwedge^kE\otimes\bigwedge^k\overline E$ with $0\leq k\leq n$, see Definition \ref{def_hr} below. In practice, $E$ is the complex cotangent space at a point $x$ of $X$ and we obtain a Hodge-Riemann cone associated with $X$. Here is our main result.

\begin{theorem} \label{th_main_1}
Let $(X,\omega)$ be a compact K{\"a}hler manifold of dimension $n$. 
Let $p,q$ be non-negative integers such that $p+q\leq n$ and 
$\Omega$ a closed smooth form of bidegree $(n-p-q,n-p-q)$ on 
$X$. Assume that $\Omega$  takes values only 
in the Hodge-Riemann cones associated 
with $X$. Then $\{\Omega\}$ satisfies the Hodge-Riemann theorem, 
the hard Lefschetz theorem and the Lefschetz decomposition theorem for the bidegree $(p,q)$. 
\end{theorem}

Roughly speaking, the hypothesis of Theorem \ref{th_main_1} says 
that at every point $x$ of $X$, we can deform continuously $\Omega$ to 
$\omega^{n-p-q}$ in a ``nice way''. However, we do not need that the 
deformation depends continuously on $x$ and a priori the deformation 
does not preserve the closedness nor the smoothness of the form.

We deduce from Theorem \ref{th_main_1} the following corollary using a
result due to Timorin \cite{Timorin}, see Proposition \ref{prop_timorin} below.

\begin{corollary} \label{cor_hr_mixed}
Let $(X,\omega)$ be a compact K{\"a}hler manifold of dimension $n$. Let
$p,q$ be non-negative integers such that $p+q\leq n$ and 
$\omega_1,\ldots,\omega_{n-p-q}$ be K{\"a}hler forms on $X$. 
Then the class $\{\omega_1\wedge\ldots\wedge \omega_{n-p-q}\}$  
satisfies the Hodge-Riemann theorem, the hard Lefschetz theorem and the Lefschetz decomposition theorem for the bidegree $(p,q)$. 
\end{corollary}

The last result was obtained by the authors in \cite{DinhNguyen}, 
see also Cattani \cite{Cattani} 
for a proof using the  theory of variations of Hodge structures  and for related results. 
It solves a problem which has been considered in some important cases by 
Khovanskii \cite{Khovanskii1,Khovanskii2}, Teissier \cite{Teissier1,Teissier2}, Gromov \cite{Gromov} and Timorin \cite{Timorin}. The reader 
will find some applications 
of the above corollary in  Gromov \cite{Gromov}, Dinh-Sibony \cite{DinhSibony}
and Keum-Oguiso-Zhang \cite{KeumOguisoZhang, Zhang}.

\section{Hodge-Riemann forms} \label{section_hr_linear}

In this section, we introduce the notion of Hodge-Riemann form in the 
linear setting and we will discuss some basic properties of these forms.

Let $E$ be a complex vector space of dimension
$n$ and $\overline E$ its conjugate space.
Denote by $V^{p,q}$  
the space  $\bigwedge^p E\otimes \bigwedge^q\overline E$  of $(p,q)$-forms with the convention that $V^{p,q}:=0$ unless
$0\leq p,q\leq n$.
Recall that a form $\omega$ in $V^{1,1}$ is a  {\it K{\"a}hler form} if 
it can
be written as
$$\omega=idz_1\wedge d\overline z_1+\cdots
+idz_n\wedge d\overline z_n$$
for some coordinate system $(z_1,\ldots, z_n)$ of $E$,
where $z_i\otimes \overline z_j$ is identified with $dz_i\wedge d\overline z_j$.

Recall also that a form $\Omega$ in $V^{k,k}$ with $0\leq k\leq n$, 
is {\it real} if
$\Omega=\overline\Omega$. Let $V^{k,k}_\R$ denote the space 
of real $(k,k)$-forms. A form 
$\Omega$ in $V^{k,k}$ is {\it positive}\footnote{There are two other notions of positivity but we will not use here.} if it is a
combination with positive coefficients of forms of type
$i^{k^2}\alpha\wedge\overline\alpha$ with $\alpha\in V^{k,0}$. So,
positive forms are real. If
$\Omega$ is positive its restriction to any
subspace of $E$ is positive. A  positive $(k,k)$-form $\Omega$ is {\it strictly positive}, if its restriction to any subspace of dimension $k$ of $E$ does not vanish. 
The powers of a K{\"a}hler form are strictly positive forms. 
Fix a K{\"a}hler form $\omega$ as above.  

\begin{definition}\rm \label{def_hr}
A real $(k,k)$-form $\Omega$ in $V^{k,k}_\R$ is said  to be 
{\it a Hodge-Riemann form for the bidegree $(p,q)$} if 
$k=n-p-q$ and if  there is a continuous deformation $\Omega_t\in
V^{k,k}_\R$ with $0\leq t\leq 1$, $\Omega_0=\Omega$ and
$\Omega_1=\omega^k$ such that 

\medskip\noindent
${\bf (*)}\quad$ the map  $\alpha\mapsto \Omega_t\wedge\omega^{2r}\wedge \alpha$ is
an isomorphism from $V^{p-r,q-r}$  to $V^{n-q+r,n-p+r}$ 

\medskip\noindent
for every $0\leq r\leq\min\{p,q\}$ and $0\leq t\leq 1$. The cone of such forms $\Omega$
is called {\it the Hodge-Riemann cone for the
bidegree $(p,q)$}. We say that $\Omega$ is {\it Hodge-Riemann} if it is a Hodge-Riemann form for any bidegree $(p,q)$ with $p+q=n-k$. 
\end{definition}

Note that the property $(*)$  for $t=1$ 
is a consequence of the linear version of the 
classical hard Lefschetz theorem. The Hodge-Riemann cone is open in $V^{k,k}_\R$ and a priori depends on the choice of
$\omega$. In practice, to check that a form is Hodge-Riemann is usually not a simple matter. We have the following result due to Timorin in \cite{Timorin}.

\begin{proposition} \label{prop_timorin}
Let $k$ be an integer such that 
$0\leq k\leq n$.
Let $\omega_1,\ldots,\omega_k$ be K\"ahler forms. Then
$\Omega:=\omega_1\wedge\ldots\wedge\omega_k$  is a Hodge-Riemann form.  
\end{proposition}

Consider a square matrix $M=(\alpha_{ij})_{1\leq i,j\leq k}$ with entries in $V^{1,1}$.  Assume that $M$ is Hermitian, i.e. $\alpha_{ij}=\overline\alpha_{ji}$ for all $i,j$. We say that $M$ is {\it Griffiths positive} if for any row vector $\theta=(\theta_1,\ldots,\theta_k)$ in $\C^k\setminus \{0\}$ and its transposed
$\transposee \theta$, $\theta M\transposee{\overline \theta}$ is a K\"ahler form. We call
{\it Griffiths cone} the set of $(k,k)$-forms in $V^{k,k}$ which can be obtained as the determinant of a Griffiths positive 
matrix $M$ as above. We are still unable to answer the following question. 

\begin{problem} Is the Griffiths cone contained in the Hodge-Riemann cone ?
\end{problem}

The affirmative answer to the question will allow us to obtain a transcendental version of the hyperplane Lefschetz theorem which is known for the last Chern class associated with a Griffiths positive vector bundle, see Voisin \cite[p.312]{Voisin}.
The Griffiths cone contains the wedge-products of K\"ahler forms (case where $M$ is diagonal) and
Proposition \ref{prop_timorin} gives the affirmative answer to this case. 

Note also that for the above problem it is enough to check the condition $(*)$ for $t=0$ and $r=0$. Indeed, 
we can consider $\Omega_t$ the determinant of the Griffiths positive matrix $M_t:=(1-t)M+tI\omega$, where $I$ is the identity matrix. It is enough to observe that $\Omega_t\wedge\omega^{2r}$ is the determinant of the Griffiths positive $(k+2r)\times (k+2r)$ matrix which is obtained by adding to $M_t$ a square block equal to $\omega$ times the identity $2r\times 2r$ matrix.

The following question is also open.

\begin{problem} \label{prob_wedge}
Let $\Omega_t$, $0\leq t\leq 1$, be a continuous family of strictly positive $(k,k)$-forms in $V^{k,k}_\R$ with $\Omega_0=\Omega$ and $\Omega_1=\omega^k$. Assume the property $(*)$ in Definition \ref{def_hr} for $r=0$ and for this family $\Omega_t$. Is $\Omega$ always a Hodge-Riemann form for the bidegree $(p,q)$ ?
\end{problem}

Note that the strict positivity of $\Omega_t$ implies the property $(*)$ for $r=\min\{p,q\}$. This is perhaps a reason to believe that the answer to the above problem is affirmative. An interesting point here is that the cone of all forms $\Omega$ as in 
Problem \ref{prob_wedge} does not depend on $\omega$.
The following result gives a partial answer to the question.

\begin{proposition} \label{th_hr_prod}
Let $\Omega_t$ be as in Problem \ref{prob_wedge}. Assume moreover that $\min\{p,q\}\leq 2$. Then $\Omega$ is a Hodge-Riemann form for the bidegree $(p,q)$.
\end{proposition}

Fix a coordinate system $(z_1,\ldots,z_n)$ of $E$ such that 
$\omega=idz_1\wedge d\overline z_1+\cdots+idz_n\wedge d\overline z_n$. So, this K\"ahler form is invariant under the natural action of the unitary group $\U(n)$. 
We will need the following lemma.

\begin{lemma} \label{lemma_prim}
Let $\alpha$ be a form in $V^{p,q-1}$ with $q\geq 2$ and
  $p+q\leq n$. Assume that for every
  $\varphi\in V^{0,1}$ we can write $\alpha\wedge\varphi=\omega\wedge\beta$
  for some $\beta\in V^{p-1,q-1}$. Then we can write
  $\alpha=\omega\wedge\gamma$ for some $\gamma\in V^{p-1,q-2}$.  
\end{lemma}
\proof
Let $M$  denote  the set  of all  forms $\alpha\in V^{p,q-1}$  satisfying the hypothesis of the lemma.
Observe that $M$  is invariant under the
action of  $\U(n)$. So, it is a linear representation
of this group. Let $P_j$ denote the primitive subspace of
$V^{p-j,q-1-j}$, i.e. the set of $\phi\in
V^{p-j,q-1-j}$ such that $\phi\wedge \omega^{n-p-q+2+2j}=0$. It is
well-known that the $P_j$ are irreductible representations of $\U(n)$
and they are not isomorphic one another, see
e.g. Fujiki \cite[Prop. 2.2]{Fujiki}. 
Moreover, we have the Lefschetz decomposition
$$V^{p,q-1}=\bigoplus_{0\leq j\leq \min\{p,q-1\}} \omega^j\wedge P_j.$$
The space $\omega^j\wedge P_j$ is also a representation of $\U(n)$
which is isomorphic to $P_j$.
Therefore, it is enough to show that $M$ does not contain $P_0$. 

Consider the form
$$\alpha:=d\overline z_2\wedge\ldots \wedge d\overline z_q\wedge d z_{q+1}\wedge \ldots \wedge d z_{p+q}.$$
A direct computation shows that $\alpha$ is a form in $P_0$.
Observe that $\alpha\wedge d\overline z_1$ does not contain any factor $dz_j\wedge d\overline z_j$. Therefore, $\alpha\not\in M$ because $\alpha\wedge d\overline z_1$ does not belong to $\omega\wedge V^{p-1,q-1}$. 
The lemma follows.
\endproof

Given non-negative integers $p,q$ such that 
$p+q\leq n$  and  a real form $\Omega$ of bidegree $(n-p-q,n-p-q)$, define the Hermitian form $Q$ by
\begin{equation*} 
Q(\alpha,\beta):= i^{q-p}(-1)^{\frac{(p+q)(p+q+1)}{2}}\ast
\Big(\alpha\wedge \overline{\beta}\wedge  \Omega \Big)\quad \mbox{for}
\quad \alpha,\beta\in V^{p,q},
\end{equation*} 
where  $\ast$ is the Hodge star operator.
Define also {\it the primitive subspace}
$$P^{p,q}:= \big\{\alpha\in  V^{p,q}:\  \alpha\wedge\Omega\wedge \omega =0 \big\}.$$

The classical Lefschetz theorem asserts that 
the wedge-product with $\omega$ defines
a surjective map from $V^{n-q,n-p}$ to $V^{n-q+1,n-p+1}$. Its kernel
is of dimension $\dim V^{p,q}-\dim V^{p-1,q-1}$. Therefore, 
if the map $\alpha\mapsto \Omega\wedge\alpha$ is injective on $V^{p,q}$,
the above primitive space 
has dimension $\dim V^{p,q}-\dim V^{p-1,q-1}$ which does not depend on 
$\Omega$.

We also need the following lemma.

\begin{lemma} \label{lemma_prim_bis}
Let  $\Omega_t$ be  a continuous family of real $(k,k)$-forms in $V^{k,k}_\R$ with 
$\Omega_0=\Omega$, $\Omega_1=\omega^k$ and $0\leq t\leq 1$.
Assume that  $\alpha\mapsto \Omega_t\wedge \alpha$ is
an isomorphism from $V^{p,q}$  to $V^{n-q,n-p}$ 
for every  $0\leq t\leq 1$ and 
$\alpha\mapsto \Omega_t\wedge\omega^{2}\wedge \alpha$ is
an isomorphism from $V^{p-1,q-1}$  to $V^{n-q+1,n-p+1}$ 
for every $0 < t \leq 1$.
If  a form $\alpha$ in $V^{p,q-1}$ (resp. $V^{p-1,q}$) satisfies $\alpha\wedge\Omega\wedge\omega=0$,  then $\alpha$ belongs to $\omega\wedge V^{p-1,q-2}$ (resp. $\omega\wedge V^{p-2,q-1}$). 
\end{lemma}
\proof
Let $V$ denote the space of forms $\beta\in V^{p,q}$ such that
$Q(\beta,\phi)=0$ for every $\phi$ in $\omega\wedge
V^{p-1,q-1}+P^{p,q}$. The hypothesis implies that $Q$ is non-degenerate. Therefore, we obtain
$$\dim \omega\wedge V^{p-1,q-1}+\dim P^{p,q}=\dim V^{p-1,q-1}+\dim V^{p,q}-\dim V^{p-1,q-1}=\dim V^{p,q},$$
and hence
$$\dim V= \dim V^{p,q}-\dim (\omega\wedge
V^{p-1,q-1}+P^{p,q})= \dim (\omega\wedge
V^{p-1,q-1}\cap P^{p,q}).$$
On the other hand, by definition of $P^{p,q}$, the space 
$\omega\wedge V^{p-1,q-1}\cap P^{p,q}$ is contained in
$V$. We deduce that these two spaces coincide.

Let $\alpha\in V^{p,q-1}$ such that $\alpha\wedge\Omega\wedge\omega=0$ (the case $\alpha\in V^{p-1,q}$ can be treated in the same way).
Fix a form $\varphi$ in $V^{0,1}$. By Lemma \ref{lemma_prim}, we only need to show that
$\alpha\wedge\varphi$ belongs to $V$. 
It is clear that $Q(\alpha\wedge\varphi,\phi)=0$ for $\phi\in
\omega\wedge V^{p-1,q-1}$. It remains to show that
$Q(\alpha\wedge\varphi,\phi)=0$ for $\phi\in P^{p,q}$. For
this purpose, it is enough to consider the case where $\varphi=d\overline z_j$
since $\{d\overline z_1,\ldots , d\overline z_n\}$ is a basis of $V^{0,1}$.

Using the continuous deformation of $\Omega$ in the hypothesis, we obtain as in Proposition \ref{prop_hr_linear} below that  the restriction of $Q$ to
$P^{p,q}$ is semi-positive. Observe that $\alpha\wedge
d\overline z_j$ is in $P^{p,q}$. Hence,
$$Q(\alpha\wedge d\overline z_j,\alpha\wedge d\overline z_j)\geq 0.$$
The sum over $j$ of $Q(\alpha\wedge d\overline z_j,\alpha\wedge d\overline z_j)$ vanishes
since $\alpha\wedge\Omega\wedge\omega=0$. We deduce that all the above
inequalities are in fact equalities. Now, since $Q$ is semi-positive
on $P^{p,q}$, by Cauchy-Schwarz's inequality, 
$Q(\alpha\wedge d\overline z_j,\phi)=0$ for $\phi\in P^{p,q}$. This
completes the proof.
\endproof 
  
\noindent{\bf Proof of Proposition   \ref{th_hr_prod}.}
Assume without loss of generality that
$q\leq p$.  
Observe that for every $\alpha$ non-zero in $V^{n-k-s,0}$ we have 
$i^{(n-k-s)^2}\alpha\wedge\overline \alpha\wedge\Omega_t\wedge\omega^s>0$. 
So, we only have to consider the case $q=2$ and to check the property $(*)$ for $r=1$. 
We will show that the map $\alpha\mapsto \Omega_t\wedge\omega\wedge\alpha$ is injective on $V^{p,1}$ and 
the map $\alpha\mapsto \Omega_t\wedge\omega^2\wedge\alpha$ is injective on $V^{p-1,1}$. The result will follow easily.

Let $\Sigma$ denote the set of $t$ satisfying the above property. By continuity, $\Sigma$ is open in $[0,1]$. Moreover, by Lefschetz theorem, it contains the point $1$. Assume that $\Sigma$ is not equal to $[0,1]$. Let $t_0<1$ be the minimal number such that $]t_0,1]\subset \Sigma$. We will show that $t_0\in\Sigma$ which is a contradiction. Up to a re-parametrization of the family $\Omega_t$, we can assume for simplicity that $t_0=0$. 

Consider a form $\alpha\in V^{p,1}$ such that $\Omega\wedge\omega\wedge\alpha=0$. We deduce from Lemma \ref{lemma_prim_bis} that $\alpha=\omega\wedge\gamma$ with $\gamma\in V^{p-1,0}$. We have $\gamma\wedge\overline\gamma\wedge\Omega\wedge\omega^2=0$. The positivity of $\Omega$ implies that $\gamma=0$ and then $\alpha=0$. 
So, the map $\alpha\mapsto\Omega\wedge\omega\wedge\alpha$ is injective on $V^{p,1}$.
By dimension reason, this map is bijective from $V^{p,1}$ to $V^{n-1,n-p}$. 
Now, we apply again Lemma \ref{lemma_prim_bis} but to $\Omega_t\wedge\omega$ instead of $\Omega_t$ and $(p,1)$ instead of $(p,q)$. We obtain as above that the map $\alpha\mapsto \Omega\wedge\omega^2\wedge\alpha$ is injective on $V^{p-1,1}$. Therefore, $0$ is a point in $\Sigma$. This completes the proof. \hfill $\square$

\bigskip

We give now fundamental properties of Hodge-Riemann forms that we will use in the next section.
We fix a norm on each space $V^{*,*}$. 
  
\begin{proposition}\label{prop_hr_linear}  
Let $\Omega$ be a form satisfying the condition $(*)$ in Definition \ref{def_hr} for $r=0,1$.
Then   
the space   $V^{p,q}$ 
splits into the $Q$-orthogonal direct sum
$$V^{p,q}=P^{p,q}\oplus \omega\wedge 
V^{p-1,q-1}$$
and the Hermitian form $Q$ is positive-definite on  
$P^{p,q}$. Moreover, for any constant $c_1>0$
large enough, there is a constant $c_2>0$ such that 
$$\|\alpha\|^2 \leq c_1 Q(\alpha,\alpha) +
c_2\|\alpha\wedge\Omega\wedge \omega\|^2\quad \mbox{for} \quad
\alpha\in V^{p,q}.$$
\end{proposition}
\proof   
The $Q$-orthogonality is obvious. 
By classical Lefschetz theorem, the wedge-product with $\omega$ defines
an injective map from $V^{p-1,q-1}$ to $V^{p,q}$. Therefore, we have 
$$\dim V^{p,q}=\dim P^{p,q}+ \dim V^{p-1,q-1}
= \dim P^{p,q}+ \dim \omega\wedge V^{p-1,q-1}.$$
On the other hand, the property $(*)$ for $r=1$ implies that 
the intersection of $P^{p,q}$
and $\omega\wedge V^{p-1,q-1}$ is reduced to 0. 
We then deduce the above decomposition of $V^{p,q}$.
Of course, this property still holds if we replace $\Omega$ with
$\Omega_t$. 

Denote by $Q_t$ and $P^{p,q}_t$ the Hermitian form and the
primitive space associated with
$\Omega_t$ which are defined as above. Since the dimension of
$P^{p,q}_t$ is constant, this space depends continuously on $t$.
By classical Hodge-Riemann theorem, $Q_1$ is positive-definite on
$P^{p,q}_1$. If $Q$ is not positive-definite on $P^{p,q}$, there is a
maximal number $t$ such that $Q_t$ is not positive-definite. The
maximality of $t$ implies that $Q_s$ is positive-definite on $P^{p,q}_s$
when $s>t$. It follows by continuity that there is an element $\alpha\in P_t^{p,q}$, $\alpha\not=0$, such that $Q_t(\alpha,\beta)=0$ for $\beta\in P_t^{p,q}$. By definition of $P^{p,q}_t$, this identity holds also for $\beta\in \omega\wedge V^{p-1,q-1}$. We then deduce that the identity holds for all $\beta\in V^{p,q}$. It follows that $\alpha\wedge\Omega_t=0$. This is a contradiction. So, $Q$ is
positive-definite on $P^{p,q}$. 

We prove now the last assertion in the proposition for a fixed 
constant $c_1$ large enough. Consider a form $\alpha\in V^{p,q}$. The first
assertion implies that we can write
$$\alpha=\beta +  \omega\wedge\gamma \quad \mbox{with}\quad 
\beta\in P^{p,q}\quad \mbox{and}\quad \gamma\in V^{p-1,q-1}$$
and we have
$$Q(\alpha,\alpha)=Q(\beta,\beta)+Q(\omega\wedge\gamma,
\omega\wedge\gamma).$$
Since the wedge-product with $\Omega\wedge\omega^2$ defines an
isomorphism between $V^{p-1,q-1}$ and $V^{n-q+1,n-p+1}$, there is a
constant $c>0$ such that
$$c^{-1}\|\gamma\wedge\Omega\wedge\omega^2\|\leq \|\gamma\| \leq c\|\gamma\wedge\Omega\wedge \omega^2\| = c\|\alpha\wedge\Omega\wedge \omega\|.$$
Therefore, there is a constant $c'>0$ such that 
$$\|\alpha\|^2\leq c'(\|\beta\|^2+\|\gamma\|^2)\leq 
c'\|\beta\|^2+c'c^2\|\alpha\wedge\Omega\wedge \omega\|^2.$$  

Finally, since $Q$ is positive-definite on $P^{p,q}$ and since $c_1>0$ is
large enough, we obtain
\begin{eqnarray*}
c'\|\beta\|^2\leq c_1Q(\beta,\beta)&=& 
c_1\big (Q(\alpha,\alpha)- Q(\omega\wedge\gamma, 
\omega\wedge\gamma)\big) \\
& \leq& c_1 Q(\alpha,\alpha)+  c_1c\|\gamma\|^2\\
&\leq& c_1 Q(\alpha,\alpha)+ c_1c^3\|\gamma\wedge\Omega\wedge \omega^2\|^2\\
&=& c_1Q(\alpha,\alpha)+ c_1c^3\|\alpha\wedge\Omega\wedge \omega\|^2.
\end{eqnarray*}
We then deduce the estimate in the proposition by taking $c_2:=c'c^2+c_1c^3$.
\endproof

\section{Lefschetz and Hodge-Riemann theorems} \label{section_hr}

In this section, we prove Theorem \ref{th_main_1}. 
Corollary \ref{cor_hr_mixed} is then deduced from that theorem and Proposition \ref{prop_timorin}. 
We will use the results of the last section for $E$ the complex cotangent space of 
$X$ at a point and $\omega$ the 
K\"ahler form on $X$. So, we can define at every point of $X$ a Hodge-Riemann cone for bidegree $(p,q)$. We now use the notation in Theorem \ref{th_main_1}. Let $ \Ec^{p,q}$ (resp. $L^2_{p,q}$) denote the spaces of smooth (resp. $L^2$) forms on $X$ of bidegree $(p,q)$. 

\begin{proposition}\label{prop_ddc}
Assume that $p,q\geq 1$.
Then, for every closed form $f\in  \Ec^{p,q}(X)$ such that  $ \{f\}\in H^{p,q}(X,\C)_\prim,$
there is a form $u\in L^2_{p-1,q-1}(X)$ such that  
\begin{equation*}
\ddc u\wedge\Omega\wedge\omega=f\wedge\Omega\wedge\omega.
\end{equation*}
\end{proposition}
\proof
Consider the subspace $H$ of $L^2_{n-p+1,n-q+1}(X)$  defined by
$$H:=\left\lbrace  \ddc \alpha\wedge\Omega\wedge\omega:\  \alpha\in  \Ec^{q-1,p-1}(X)\right\rbrace$$
and the linear form $h$ on $H$ given by
$$h(\ddc \alpha\wedge\Omega\wedge\omega):= (-1)^{p+q+1}\int_{X}\alpha\wedge f\wedge  \Omega\wedge\omega. $$
We prove that  $h$ is a well-defined bounded linear form with respect to the $L^2$-norm restricted to $H$.

We claim that there is a constant $c>0$ such that 
$$\|\ddc \alpha\|_{L^2}\leq c \|\ddc \alpha\wedge\Omega\wedge\omega\|_{L^2}.$$
Indeed, we use the inequality in Proposition \ref{prop_hr_linear} applied to $\ddc\alpha$ instead of $\alpha$ and the complex cotangent spaces of $X$ instead of $E$. Since $X$ is compact, we can find common constants $c_1$ and $c_2$ for all cotangent spaces. We then integrate over $X$ and obtain
$$\|\ddc\alpha\|_{L^2}\leq c_1Q(\ddc\alpha,\ddc\alpha)
+c_2\|\ddc\alpha\wedge \Omega\wedge\omega\|_{L^2}^2,$$
where $Q$ is defined in Section \ref{introduction}.  Using Stokes' formula, we obtain
 $$Q(\ddc \alpha,\ddc\alpha)= i^{q-p}(-1)^{(p+q)(p+q+1)\over 2}
     \int_X \ddc \alpha\wedge\ddc\overline\alpha\wedge\Omega=0.$$
We then deduce easily the claim.

Now, by hypothesis the smooth form $f\wedge\Omega\wedge\omega$ is exact. Therefore, 
there is a form $g\in  \Ec^{n-q,n-p}(X)$ such that 
$$\ddc g =f\wedge\Omega\wedge\omega,$$
see e.g. \cite[p.41]{bdi}.
Using again Stokes' formula and the above claim, we obtain  
 \begin{eqnarray*}
 \Big|\int_X\alpha\wedge f\wedge  \Omega\wedge\omega\Big|
& = & \Big|\int_X\alpha\wedge \ddc g\Big|=\Big|\int_X\ddc \alpha\wedge  g\Big| \\
 & \leq &   \|g\|_{L^2}\|\ddc\alpha\|_{L^2}  \leq 
 c\|g\|_{L^2}\|\ddc \alpha\wedge\Omega\wedge\omega\|_{L^2}.
 \end{eqnarray*}
It follows that $h$ is a well-defined form whose norm in $L^2$ is bounded by $c\|g\|_{L^2}$.

By Hahn-Banach theorem, we can extend $h$ to a bounded linear form  
on $L^2_{n-p+1,n-q+1}(X)$.
Let $u$ be a form in  $L^2_{p-1,q-1}(X)$  that represents $h$. It follows from the definition of $h$ that \begin{equation*} 
\int_X u\wedge \ddc\alpha \wedge\Omega\wedge\omega = (-1)^{p+q+1}\int_X\alpha\wedge f\wedge  \Omega\wedge\omega = -\int_X f\wedge\alpha\wedge\Omega\wedge\omega
\end{equation*}
for all test forms $\alpha\in  \Ec^{q-1,p-1}(X)$. The form $u$ satisfies the proposition.
\endproof

We have the following result.

\begin{proposition}\label{prop_ddc_smooth}
Let $u$ be as in  Proposition \ref{prop_ddc}. Then 
there is a form $v\in  \Ec^{p-1,q-1}(X)$ such that  $\ddc v=\ddc u$.
 \end{proposition}
 \begin{proof}
 We can assume without loss of generality that  $p\leq q$. The idea is to use the ellipticity of the Laplacian operator associated with $\dbar$ and a special inner product on $ \Ec^{p,q}(X)$. 
 We first construct this inner product. Fix an arbitrary Hermitian metric
 on the vector bundle $\bigwedge^{r,s}(X)$ of differential 
 $(r,s)$-forms on $X$  with $(r,s)\not=(p,q)$ and denote by 
 $\langle\cdot,\cdot\rangle$ the associated inner product on
 $ \Ec^{r,s}(X)$.

Using  the first assertion in Proposition \ref{prop_hr_linear}, for any $\alpha,\alpha'\in  \Ec^{p,q}(X)$, we can write in a unique way
$$\alpha=\beta+\omega\wedge\gamma \quad \mbox{and}\quad \alpha'=\beta'+\omega\wedge\gamma'$$ 
 with $\beta,\beta'\in \Ec^{p,q}(X)$ and $\gamma,\gamma'\in  \Ec^{p-1,q-1}(X)$  such that $\beta\wedge \Omega\wedge \omega=0$ 
 and  $\beta'\wedge \Omega\wedge \omega=0$. 
Define an inner product $\left <\cdot, \cdot\right>$ on $ \Ec^{p,q}(X)$  by setting
 $$\langle \alpha, \alpha'\rangle := Q(\beta,\beta')+\langle \gamma,\gamma'\rangle
 =Q(\alpha,\beta')+\langle \gamma,\gamma'\rangle.$$
This inner product is associated with a Hermitian metric on $\bigwedge^{p,q}(X)$. 

Using the positivity of $Q$ given in Proposition \ref{prop_hr_linear}, we see that
$\langle\cdot, \cdot\rangle$ defines a Hermitian metric on $ \Ec^{p,q}(X)$. Consider now the norm 
$\|\alpha\|:=\sqrt{\langle\alpha, \alpha\rangle}$. Then there is a constant $c>0$ such that
$$c^{-1} \big (\|\beta\|_{L^2}+ \|\gamma\|_{L^2} \big) \leq  \|\alpha\|\leq c \big (\|\beta\|_{L^2}+ \|\gamma\|_{L^2} \big).$$

Consider the $(p,q)$-current $h:=\ddc u-f$ which belongs to a Sobolev space. We have
 $$\dbar h=0,\quad \partial h=0\quad  \text{and}\quad   h\wedge \Omega\wedge \omega=0.$$
 The last identity says that if we decompose $h$ as we did above for $\alpha,\alpha'$, the second component in the decomposition vanishes. Therefore, $\langle\dbar\alpha, h\rangle= Q(\dbar\alpha, h)$ for any form $\alpha\in  \Ec^{p,q-1}(X)$. Using Stokes' formula, we obtain 
 $$\langle\dbar\alpha, h\rangle= Q(\dbar\alpha, h)=
 i^{q-p}(-1)^{p+q-1+{(p+q)(p+q+1)\over 2}}\int_X\alpha\wedge \dbar\overline h\wedge\Omega =0.$$
If $\dbar^*$ is the adjoint of $\overline{\partial}$
with respect to the considered inner products, we deduce that $\dbar^* h=0$.
 On the other hand,  $\dbar h=0$.
 Therefore, $h$ is a harmonic current with respect to the Laplacian operator
 $\dbar\dbar^*+\dbar^*\dbar$, see Section 5 in \cite[Chap. IV]{Wells}.
 Consequently, by elliptic regularity, 
 $h$ is smooth, see e.g. Theorem 4.9 in \cite[Chap. IV]{Wells}). Hence, $\ddc u$ is smooth.
 We deduce the existence of
 $v\in \Ec^{p-1,q-1}(X)$ such that  $\ddc v=\ddc u,$  see e.g. \cite[p.41]{bdi}.
 \end{proof}
 
\noindent
{\bf  End of the proof of Theorem \ref{th_main_1}.}
Let  $f$   be a closed form in $  \Ec^{p,q}(X)$ such that  $ \{f\}\in H^{p,q}(X,\C)_\prim$.
We first show that   $Q(\{f\},\{f\})\geq 0$. Let $v$ be the smooth $(p-1,q-1)$-form  given by  Proposition \ref{prop_ddc_smooth}.  
Then we have
$$(f-\ddc v)\wedge  \Omega\wedge\omega=0.$$
Here, we should replace $\ddc v$ with 0
when either $p=0$ or $q=0$.
Using Proposition \ref{prop_hr_linear} to each point of $X$,
after an integration on $X$, we obtain 
$$ i^{q-p}(-1)^{\frac{(p+q)(p+q+1)}{2}  }\int_{X} (f-\ddc v)\wedge (\overline{f}-\ddc \overline{v})\wedge \Omega \geq 0.$$
Using Stokes' formula and that $f$ is closed, we obtain 
$$\int_{X} f \wedge \overline{f} \wedge \Omega= \int\limits_{X} (f-\ddc v)\wedge (\overline{f}-\ddc \overline{v})\wedge \Omega .$$
Therefore, $Q(\{f\},\{f\})\geq 0$. The equality occurs if and only if $f=\ddc v$,  i.e.
$\{f\}=0.$  Hence, $\{\Omega\}$ satisfies the Hodge-Riemann theorem for the bidegree $(p,q)$.  

We deduce that the map $\{\alpha\}\mapsto \{\alpha\}\smile \{\Omega\}$ is injective on $H^{p,q}(X,\C)_\prim$. 
If $\{\alpha\}$ is a class in $H^{p,q}(X,\C)$ such that $\{\alpha\}\smile \{\Omega\}=0$,  $\{\alpha\}$ is a primitive class and hence $\{\alpha\}=0$. Therefore, $\{\Omega\}$ satisfies the hard Lefschetz theorem for the bidegree $(p,q)$. 

The classical hard Lefschetz theorem implies that  $\{\alpha\}\mapsto \{\alpha\}\smile \{\omega\}$ is an 
injective map from $H^{p-1,p-1}(X,\C)$ to $H^{p,q}(X,\C)$. Therefore,
$$\dim \{\omega\}\smile H^{p-1,q-1}(X,\C)=\dim H^{p-1,q-1}(X,\C).$$
This Lefschetz theorem also implies that 
$\{\alpha\}\mapsto \{\alpha\}\smile \{\omega\}$ is a surjective map from $H^{n-q,n-p}(X,\C)$ to  $H^{n-q+1,n-p+1}(X,\C)$. 
This together with the hard Lefschetz theorem for $\{\Omega\}$ yield 
\begin{eqnarray*}
\dim H^{p,q}(X,\C)_\prim & = & \dim H^{p,q}(X,\C)-\dim H^{n-q+1,n-p+1}(X,\C) \\
& = & \dim H^{p,q}(X,\C)-\dim H^{p-1,q-1}(X,\C)\\
& = & \dim H^{p,q}(X,\C)-\dim \{\omega\}\smile H^{p-1,q-1}(X,\C).
\end{eqnarray*}

The hard Lefschetz theorem can also be applied to $\{\Omega\wedge\omega^2\}$ 
and to the bidegree $(p-1,q-1)$. We deduce that 
the intersection of $\{\omega\}\smile H^{p-1,q-1}(X,\C)$ and $H^{p,q}(X,\C)_\prim$ is reduced to 0.
This together with the above dimension computation  gives us 
the following decomposition into a direct sum
$$H^{p,q}(X,\C)=\{\omega\}\smile H^{p-1,q-1}(X,\C)+H^{p,q}(X,\C)_\prim.$$
Finally, the previous decomposition is orthogonal with respect to $Q$ by definition of primitive space. So, $\{\Omega\}$ satisfies the Lefschetz decomposition theorem.
 \hfill $\square$

 \begin{remark} \rm
 In order to obtain the Hodge-Riemann theorem and the hard Lefschetz theorem (resp. the Lefschetz decomposition), it is enough to assume the property $(*)$ in Definition \ref{def_hr} for $r=0,1$ (resp. $r=0,1,2$). 
 When $(*)$ is satisfied for all $r$, we can apply inductively these theorems to $\Omega\wedge\omega^{2r}$ and then obtain the signature of $Q$ on $H^{p,q}(X,\C)$.  
 \end{remark}

\noindent
T.-C. Dinh, UPMC Univ Paris 06, UMR 7586, Institut de
Math{\'e}matiques de Jussieu, 4 place Jussieu, F-75005 Paris, France.\\ 
{\tt dinh@math.jussieu.fr}, {\tt http://www.math.jussieu.fr/$\sim$dinh}

\medskip

\noindent
V.-A.  Nguy{\^e}n,
Math{\'e}matique - B{\^a}timent 425, Universit{\'e} Paris-Sud, 
91405 Orsay, France.\\ 
{\tt  VietAnh.Nguyen@math.u-psud.fr}.


\small

\begin{thebibliography}{99}



\bibitem{BerndtssonSibony}
B. Berndtsson and  N. Sibony, The $\bar{\partial }$-equation on a positive current, {\it Invent. Math.}, {\bf 147:2} (2002), 371-428.

\bibitem{bdi}  J. Bertin, J.-P. Demailly,   L. Illusie and  C. Peters,
{\it  Introduction {\`a} la th{\'e}orie de Hodge}, Panoramas et Synth{\`e}ses, {\bf 3},  Soci{\'e}t{\'e} Math{\'e}matique de France, Paris, 1996. 



\bibitem{Cattani}
E. Cattani, Mixed Lefschetz theorems and Hodge-Riemann bilinear relations, {\it Int. Math. Res. Not.}, {\bf 10} (2008), 20 pp.



\bibitem{Demailly}
J.-P. Demailly, {\it Complex analytic geometry}, available at \\
{\tt www.fourier.ujf-grenoble.fr/$\sim$demailly}. 
 
 \bibitem{DinhNguyen}
T.-C. Dinh and V.-A. Nguyen, The mixed Hodge-Riemann bilinear relations for compact 
K{\"a}hler manifolds, {\it Geom. Funct. Anal.}, Vol. {\bf 16} (2006), 838-849. 

\bibitem{DinhSibony}
T.-C. Dinh and N. Sibony,
Groupes commutatifs d'automorphismes d'une vari{\'e}t{\'e} k{\"a}hl{\'e}rienne compacte, 
{\it Duke Math. J.}, {\bf 123} (2004), no. 2, 311-328.

 \bibitem{Fujiki}
A. Fujiki, On the de Rham cohomology group of a compact K{\"a}hler symplectic manifold, 
Algebraic geometry, Sendai, 1985, 105-165, {\it Adv. Stud. Pure Math.}, {\bf 10}, North-Holland, 
Amsterdam, 1987. 


 \bibitem{GriffithsHarris}
 Ph. Griffiths and J. Harris, {\it Principles of algebraic geometry}, Reprint of the 1978 original. 
Wiley Classics Library. John Wiley $\&$ Sons, Inc., New York, 1994. 

\bibitem{Gromov}
M. Gromov, Convex sets and K{\"a}hler manifolds, {\it Advances in Differential Geometry and Topology},  World Sci. Publishing, Teaneck, NJ (1990), 1-38.



\bibitem{KeumOguisoZhang}
J. Keum, K. Oguiso and D.-Q. Zhang, 
Conjecture of Tits type for complex varieties and theorem of Lie-Kolchin type for a cone, {\it Math. Res. Lett.}, {\bf  16} (2009), no. 1, 133-148.

\bibitem{Khovanskii1}
A.G. Khovanskii, Newton polyhedra, and the genus of complete intersections, {\it Funktsional. Anal. i Prilozhen},
 {\bf 12:1} (1978), 51-61 (in Russian). 
 
 \bibitem{Khovanskii2}
A.G. Khovanskii, The geometry of convex polyhedra and algebraic geometry, {\it Uspehi Mat. Nauk.}, {\bf  34:4} (1979), 160-161 (in Russian).


\bibitem{Teissier1}
B. Teissier, Du th{\'e}or{\`e}me de l'index de Hodge aux in{\'e}galit{\'e}s isop{\'e}rim{\'e}triques, {\it C.R. Acad. Sci. Paris},  
S{\'e}r. A-B {\bf 288:4} (1979), A287-A289. 

\bibitem{Teissier2}
B. Teissier, Vari{\'e}t{\'e}s toriques et polytopes, {\it Bourbaki Seminar 1980/81}, Springer Lecture Notes in Math., {\bf 901} (1981), 71-84.

\bibitem{Timorin}
V.A. Timorin, Mixed Hodge-Riemann bilinear relations in a linear context, {\it Funct. Anal. Appl.}, {\bf 32:4} (1998), 268-272.

 \bibitem{Voisin}
 C. Voisin, {\it Th{\'e}orie de Hodge et g{\'e}om{\'e}trie alg{\'e}brique complexe}, Cours Sp{\'e}cialis{\'e}s, {\bf 10}, Soci{\'e}t{\'e} Math{\'e}matique de France, Paris, 2002. 

 \bibitem{Wells}
R.O. Wells, {\it Differential analysis on complex manifolds}, Second edition. Graduate Texts in 
Mathematics, {\bf 65}, Springer-Verlag, New York-Berlin, 1980. 

\bibitem{Zhang}
D.-Q. Zhang, A theorem of Tits type for compact K{\"a}hler manifolds, {\it Invent. Math.}, {\bf 176} (2009), no. 3, 449-459.

\end{thebibliography}
\end{document}